\documentclass[reqno,10pt]{amsart}

\usepackage{amsthm}
\usepackage{amssymb,amsfonts}
\usepackage{mathtools}

\usepackage{enumerate}
\usepackage{mathrsfs}
\usepackage{mathtools}
\usepackage{xcolor}
\usepackage{hyperref}

\usepackage{amsthm}
\theoremstyle{definition}
\newtheorem{thm}{Theorem}[section]
\theoremstyle{definition}

\newtheorem{cor}[thm]{Corollary}
\newtheorem{prop}[thm]{Proposition}
\newtheorem{lem}[thm]{Lemma}
\newtheorem{conj}[thm]{Conjecture}

\newtheorem{defn}[thm]{Definition}

\newtheorem{rem}{Remark}[section] 

\def\C{{\mathbb C}}
\def\Z{{\mathbb Z}}
\def\D{{\mathbb D}}

\def\Q{{\mathbb Q}}

\def\F{{\mathbb F}}

\DeclareMathOperator{\End}{End}

\DeclareMathOperator{\np}{np}

\DeclareMathOperator{\Hom}{Hom}
\DeclareMathOperator{\Mod}{mod}

\DeclareMathOperator{\ord}{ord}

\DeclareMathOperator{\Gal}{Gal}

\DeclareMathOperator{\Sp}{Spec}
\DeclareMathOperator{\Spf}{spf}
\DeclareMathOperator{\disc}{disc}

\def\o{{\mathcal O}}
\def\G{{\mathscr G}}
\def\F{{\mathbb F}}
\def\F{{\mathbb F}}

\def\p{{\mathfrak p}}

\DeclareMathOperator{\tr}{tr}

\DeclareMathOperator{\coker}{coker}

\usepackage{tikz-cd}

\hypersetup{ colorlinks=true,
linkcolor=blue
}

\begin{document}
\title{size of isogeny classes of abelian varieties of lubin-tate type}

\author{Tejasi Bhatnagar}

\vspace{-1cm}

\email{tejasi.bhatnagar@gmail.com}

\address{Department of Mathematics\\ University of Wisconsin Madison\\ Van Vleck Hall, 480 Lincoln Drive\\ Madison, Wisconsin\\ United States.}

\begin{abstract}
We prove a lower bound for the size of the isogeny class of a simple abelian variety over a finite field with commutative endomorphism ring in the Lubin-Tate case. Moreover, based on the expected size of the isogeny classes in the Newton stratum, we conjecture that this lower bound is sharp.
\end{abstract}
\maketitle
\section{Introduction}
\subsection{The Newton stratum}Throughout this paper, we denote by $\F_q$ the finite field with $q$ elements where $q$ is a power of an odd prime. For any positive integer $n$, we denote by $\F_{q^n}$ the unique finite field extension of $\F_q$ of degree $n$. We will denote the algebraically closed field $\overline{\F}_q$ by $k$. Let $\mathcal{A}_{g,1}$ denote the moduli space of $g$ dimensional principally polarized abelian varieties in characteristic $p.$ This moduli space  is often studied by describing it as a disjoint union of locally closed subvarieties. One way to do this is to introduce a ``stratification" using the Newton polygon which is an invariant in characteristic $p$. Let $A$ be a principally polarized abelian variety over $\F_q$ and let $N(A)$ be the associated Newton stratum. We know that this is a locally closed subvariety of $\mathcal{A}_{g,1}$. Then  $N(A)(\F_{q^n})$ consists of the points of $\mathcal{A}_{g,1}$ defined over $\F_{q^n}$ that have the same Newton polygon as that of $A$. In particular, it contains all the elements in the isogeny class of $A$ as defined below.
\begin{defn}(Isogeny class) We define $I(A, \F_{q^n})$, to be the set of isomorphism classes of principally polarized abelian varieties $y\in  \mathcal{A}_{g,1}(\F_{q^n})$ isogenous to $A$ over $k$. We call this the isogeny class of $A$ defined over $\F_{q^n}.$
\end{defn}
In his paper \cite{foliations}, Frans Oort gives an ``almost product" structure to the Newton stratum, which we call a ``foliation". The foliation consists of two ``transversal" subvarieties known as the  ``central leaf" and the ``isogeny leaf". The central leaf contains all the abelian varieties over $\F_{q^n}$ whose $p$-divisible group is isomorphic to that of $A$ over $k$. In particular, it contains all the abelian varieties that are isogenous to $A$ by prime-to-$p$ isogenies, while the ``isogeny leaf" consists of the abelian varieties that are isogenous to $A$ by a certain $p$-power isogeny. Using this structure, Shankar and Tsimerman (see \cite{tsimshank}, Section 2) find heuristics for the size of isogeny classes of abelian varieties. We briefly describe this.
\subsection{Heuristics for the size of the isogeny class.}\label{heuristic}
Let $c(A)$ and $\iota(A)$ denote the central leaf and the isogeny leaf of $A$ respectively. 
\begin{defn}(Weil $q$-number of an abelian variety) Let $A$ be as above. Then we know that any eigenvalue $\alpha_A$ of its geometric Frobenius is an algebraic integer such that for every embedding $\psi: \Q(\alpha_A)\hookrightarrow \C$, we have $|\psi(\alpha_A)| = \sqrt{q}.$ Such an algebraic integer is called a Weil $q$-number. We recall that two Weil $q$-numbers $\alpha$ and $\alpha^{\prime}$ are conjugate if there is an isomorphism between $\Q(\alpha)\rightarrow \Q(\alpha^{\prime})$ taking $\alpha$ to $\alpha^{\prime}$.  The map  $A\mapsto \alpha_A$ from the isogeny classes of abelian varieties to Weil $q$-numbers (up to conjugation) is a bijection for simple abelian varieties by the Honda-Tate theory.

\end{defn}

By \cite{tsimshank}, Lemma 2.2, we know that the number of Weil $q^n$-numbers with a fixed Newton polygon is roughly $O(q^{n \dim c(A)/2})$ which gives us an estimate for the number of $\F_{q^n}$ isogeny classes in the Newton stratum. As one should expect, the size of $I(A, \F_{q^n})$ should be $O(q^{n(\dim c(A)/2+\dim \iota(A))})$ to account for roughly $q^{n \dim N}$ points of $N(\F_{q^n})$. The dimensions of the central leaves and Newton strata can be computed explicitly and only depend on the Newton polygon of $A$. This was proved in a follow-up paper by Oort (see \cite{dimfol}, Sections 5 and 6).

\subsection{Previous work} The question of counting isogeny classes is trivial for supersingular abelian varieties. Indeed, all the abelian varieties in the supersingular stratum are isogenous over $k$. Therefore, the Newton stratum is the entire isogeny class in this case. One of the first results goes back to the upper bound for elliptic curves by Lenstra \cite{lenstra} and subsequent work provides mainly lower bounds in higher dimensions. The lower bounds of $I(A, \F_{q^n})$ for the ordinary and the almost ordinary cases is worked out in \cite{tsimshank} and \cite{AA} respectively. In another paper \cite{yume} with Yu Fu, we estimate isogeny classes of certain abelian varieties with real multiplication. In this case, we provide upper and lower bounds. The second author of the paper has work in this direction as well. For example, see \cite{yufuisog} for the case of a product of a supersingular elliptic curve with an ordinary one. 
\vspace{2mm}

 In this paper, we further extend the work for the Newton stratum of \textit{Lubin-Tate type}.  
\begin{defn}
A $p$-divisible group is of Lubin-Tate type if it is a direct sum of $p$-divisible groups such that the local-local part of each summand or that of its dual has dimension one. 
\end{defn}

As a consequence, the Newton polygon of each $p$-divisible group in the summand has slope $1/m$ or $1-1/m$ for some positive integer $m>1$. In this paper, we work with principally polarized abelian varieties whose Newton polygon has slope $1/g$ and $(g-1)/g$. Since $g=2$ amounts to the abelian variety being supersingular, we will take $g>2$ for the rest of the paper.
 \subsection{The main theorem}
Let $A\in  \mathcal{A}_{g,1}(\F_{q})$ be a simple abelian variety with the Weil $q$-number $\alpha_A$. For convenience, we drop the subscript and the dependence on $A$ will be implicit. We assume that $1,\alpha =\alpha_1,$ and its conjugates $\alpha_2, \dots \alpha_g$, are multiplicatively independent. Let the $p$-divisible group $A[p^{\infty}
             ]$ of $A$ be of Lubin-Tate type, that is,  its Newton polygon has slopes $1/g$ and $(g-1)/g$.  We further assume that $K := \Q(\alpha)$ is the geometric endomorphism algebra of $A$ such that $K\otimes \Q_p$ is a direct sum $K^{\prime}\oplus K^{\prime}$ where $K^{\prime}$ is the degree $g$ unramified field extension of $\Q_p.$

For any positive integer $n$, let $|I(A, \F_{q^n})|$ denote the size of the set $I(A, \F_{q^n})$. We prove the following lower bound for  $|I(A, \F_{q^n})|$. 

\begin{thm}\label{mainthm}
For a set of positive integers $n$ with positive natural density, we have the following lower bound  $$|I(A,\F_{q^n})|\geq q^{n(\frac{(g+1)(g-2)}{4}+1+o(1))}.$$

\end{thm}

\begin{conj}
For a set of positive integers $n$ with positive natural density, we have the following estimate $$|I(A,\F_{q^n})|= q^{n(\frac{(g+1)(g-2)}{4}+1+o(1))}.$$
\end{conj}
We note that $q$ and $g$ are fixed in the above estimates. The dimension of the Newton stratum  is $g(g-1)/2$ and that of the central leaf $c(A)$  is $(g+1)(g-2)/2$. Therefore, the bound in Theorem \ref{mainthm} agrees with the expected estimate. 

As mentioned in \cite{tsimshank}, one way to prove the upper bound would be to estimate the class groups of \textit{all} the endomorphism rings of the isogeny class. We might do this by finding a parametrisation of non-maximal orders of $K$ containing the smallest possible endomorphism ring. This seems out of reach. However, methods involving orbital integrals have provided upper bounds in certain cases as well. See for instance \cite{orbit}.
\subsection{Comparison of methods from previous cases} The methods used to find bounds in the cases of ordinary and almost ordinary abelian varieties was a consequence of a classification theorem due to Deligne for ordinary abelian varieties and by Oswal-Shankar in the almost ordinary case. In both the classification results, a crucial step is to find ``canonical" CM lift(s) of abelian varieties over $\F_q$ to characteristic zero. Our proof strategy as described below is different from the previous cases in that we do not use lifting methods but count in characteristic zero and then deduce the count mod $p.$
\subsection{Overall strategy} Let $W(-)$ denote the functor of Witt vectors. Let $L = W(\F_q)(1/p)$ be the unramified extension of $\Q_p$ with residue field $\F_q$, and $L^{\prime}$ be $W(k)(1/p)$,  the maximal unramified extension of $\Q_p.$ Let $v$ be the unramified place of $L^{\prime}$ over $p.$

\subsubsection*{Counting in characteristic zero} We count in characteristic zero (over $L^{\prime})$ as follows. We fix an abelian variety $\widetilde{A}$ over $L$ such that the geometric endomorphism algebra of $\widetilde{A}$ and that of its reduction $A$, is the CM field $K$. We further fix a CM-type $(K, \Phi)$ on the endomorphism algebra of $\widetilde{A}$. As a first step, we note that there are only finitely many  abelian varieties over $L^{\prime}$ that reduce to $A$ modulo $v$. Indeed, the number of such abelian varieties over $L^{\prime}$ with endomorphism algebra $K$ depends on the number of CM-types of $K$. There are only finitely many such. Furthermore, once we fix a CM type on $K$, we get an injective map from the isogeny class over $L^{\prime}$ to the isogeny class in char $p.$ This is proved in Section \ref{sec1}, Proposition \ref{inj}.  Therefore, it is enough to count the number of abelian varieties over $L^{\prime}$ with a fixed CM type. Let $n$ be any positive integer. Now, to obtain a lower bound for $|I(A, \F_{q^{n}})|$, we reduce the problem to counting the isomorphism classes of abelian varieties in $I(\widetilde{A}, L^{\prime})$ whose endomorphism algebra $K$ contains a certain fixed endomorphism ring $R_n$. To that end, we give an injection from $Cl(R_n)$, the class group of $R_n$, to $I(\widetilde{A},L^{\prime})$. We note that the ring $R_n$ depends on $n$ and is characterized by certain local conditions at the prime $p$, which we compute in Section \ref{localcon}. 
It is crucial to take into account these conditions as they help in estimating the class number of $R_n$. Moreover, these local conditions ensure that all the abelian varieties in the set $I(\widetilde{A},L^{\prime})$ with endomorphism ring $R_n$ descend to $W(\F_{q^n})(1/p)$ (Lemma \ref{descend}). Consequently, this gives us a sharp lower bound for $|I(A, \F_{q^n})|$ once we reduce mod $p.$  
    \subsubsection*{Reducing mod $p$} After we carry out the count in characteristic zero, we reduce mod $p$ all isogenous abelian varieties over $W(\F_{q^n})(1/p)$ with endomorphism ring $R_n$. In Section \ref{centralleaf}, we show that all these abelian varieties land in exactly one central leaf in the Newton stratum that contains $A$. We can then accommodate the $p$-power isogenies coming from the isogeny leaf of $A$ by multiplying the count in characteristic zero with $q^{n\dim \iota(A)}$. Hence, it is sufficient to prove that $|I(\widetilde{A}, L^{\prime})|\geq q^{n\dim c(A)/2}$ to match our upper bound with the heuristic in Section \ref{heuristic}. 
    \begin{rem}
More generally, by \cite{KLSS} Theorem 3.6, it is known that the set of $\overline{\F}_q$-points of any Newton stratum that have CM lifts lie in finitely many central leaves. We can use this result to apply this strategy of counting points in characteristic zero and then reducing mod $p$ to accommodate the points from the isogeny leaf separately, to other non-ordinary Newton strata as well.
\end{rem}
\subsection{Organisation of the paper}
In Section \ref{sec1} we prove that the reduction mod $p$ map from the isogeny class in characteristic zero to finite fields is injective. In Section \ref{localcon} and \ref{propring}, we compute the local conditions for the endomorphism ring that we use to count and certain properties of abelian varieties with that particular endomorphism ring. In Section \ref{estimates}, we give an injection between the isogeny class and the class group of the order.  In Section \ref{count}, we find bounds for the (unpolarized) isogeny class by estimating the class group of the order. Finally, Section \ref{polarization} is dedicated to restricting our count to principally polarized abelian varieties.

\subsection{The unpolarized count} To make use of the theory of moduli spaces in characteristic $p$, we need to restrict our question to  principally polarized abelian varieties. Indeed, in the polarized case, the completion of $\mathcal{A}_{g,1}$ over $\F_p$ at an $\F_q$-point $A$, denoted by $\mathcal{A}_{g,1}^{\backslash A} = \Spf(Y)$ represents the deformation functor of $A$ as a formal scheme. If we consider the universal deformation $\mathscr{A}$ over $\Spf(Y)$, then that can be algebraised to give us an abelian scheme over $\Sp (Y)$ since $\mathscr{A}$ is principally polarized. Let $\mathscr{G}$ denote its $p$-divisible group over $\Sp(Y)$ and $\mathscr{G}_s$ the pullback of $\mathscr{G}$ to a point $s\in \Sp(Y).$ Then the polarized central leaf of $A$ is defined to be all $s\in \Sp (Y)$ with residue field $k(s)$ such that the $p$-divisible group $\mathscr{G}_s\otimes \overline{k(s)}$ of the point is isomorphic to that of $A$ over $\overline{k(s)}.$ In the unpolarized case, the universal deformation $\mathscr{A}^{\np}$  over $\Spf(X)$ of $A$ need not algebraise. Here $\Spf(X)$ is the unpolarized deformation space of $A$. However, its $p$-divisible group over $\Spf(X)$ comes from a $p$-divisible group $\mathscr{G}^{\np}$ over $\Sp(X)$ (see \cite{dejong}, Lemma 2.4.4). The notion of unpolarized central leaf still makes sense. Suppose $x_0$ is the closed point of $\Sp(X)$ that represents $A$. Then the unpolarized central leaf of $A$ is the set of all points $s\in \Sp(X)$ such that $\mathscr{G}_s^{\np}\otimes \overline {k(s)}$ is isomorphic to the $p$-divisible group of the special fibre $\mathscr{G}^{\np}_{x_0}$ over $\overline {k(s)}$. This represents a formal object in the local deformation space of $A$. The dimension of the unpolarized central leaf is also determined by the Newton polygon of $A$ and can be computed using the methods outlined in \cite{dimfol}, Section 4. Now, without referring to moduli spaces, we can ask for the count of isogeny classes of abelian varieties over finite fields without any conditions of polarization. This count in characteristic zero will be the same as computed in Section \ref{unpolri} before we take polarization into account. This is recorded in Corollary \ref{unpol}. We note that analogously to the case of polarization, the lower bound for this estimate is of order $\mathcal{O}(q^{\dim c_{\np}/2})$ where $c_{\np}$ is the unpolarized central leaf in the local deformation space of $A.$ Therefore, the count over finite fields will be greater than or equal to $q^{n(\dim (c_{\np}/2+\iota_{\np})+o(1))}$ where $\iota_{\np}$ is the unpolarized isogeny leaf of the $p$-divisible group of $A$.

\subsection*{Acknowledgements}
 We are grateful to Ananth Shankar for suggesting this project and for many helpful discussions regarding this paper. We thank Asvin G. for a helpful conversation regarding Lemma \ref{isogunrami} of the paper. We thank the referee for a careful reading of the draft and  greatly improving the exposition of the paper. The author was partially supported by the NSF grant DMS-2100436.

\section{Injectivity of the special fibre map}\label{sec1}
 Let $\widetilde{A}$ be a simple abelian variety over $L = W(\F_q)(1/p)$ with good reduction. Let $A$ denote the reduction mod $p$ of its N\'eron model. We further assume that $\End^0(A) = K = \End^0(\widetilde{A})$. Choosing an embedding of $L$ into $\C$ induces a primitive CM type $\Phi$ on $K$ determined by $\widetilde{A}$, base changed to $\C$. Fixing one such CM type $\Phi$ on $K$ determines the isogeny class of $\widetilde{A}$ over an algebraically closed field of characteristic $0$ (see Proposition 1.5.4.1 in \cite{complexmultlift}). In fact, we have the following lemma: 
  \begin{lem}\label{isogunrami}
Let $\widetilde{A}$ (base changed to $L^{\prime} = W(k)(1/p))$ be as above and $\widetilde{B}$ be another abelian variety defined over $L^{\prime}$ such that they have good reduction at the place $v$ of $L^{\prime}$. Suppose that they induce the same CM-type over the algebraically closed field $\overline\Q_p$. That is, they are isogenous over $\overline\Q_p$. Then they are isogenous over $L^{\prime}$.
\end{lem}

\begin{proof}
 Let $\bar{L}^{\prime}$ denote the algebraic closure of $L^{\prime}$ and $I_{\overline{v}}$ denote the inertia group at the place $\overline{v}$ where $\overline{v}$ is the extension $v$ to $\bar{L}^{\prime}$. Let $\widetilde{A}$ and $\widetilde{B}$ be isogenous over a finite extension of $L^{\prime}$. Since $\widetilde{A}$ and $\widetilde{B}$  both have good reduction, we see that $\Hom_{\bar{L^{\prime}}}(\widetilde{A}, \widetilde{B})\hookrightarrow \Hom_{k}(\widetilde{A}_k, \widetilde{B}_k)$. Here $\widetilde{A}_k$ (respectively $\widetilde{B}_k)$ denote the reduction mod $p$ of the N\'eron model of $\widetilde{A}$ (respectively $\widetilde{B}$).  This shows that the Galois action of $\Gal(\bar{L}^{\prime}/L^{\prime})$ on $\Hom_{\bar{L^{\prime}}}(\widetilde{A}, \widetilde{B})$ factors via $\Gal(k/\F_q)$. That is, $I_{\overline{v}}$ acts trivially on $\Hom_{\bar{L^{\prime}}}(\widetilde{A}, \widetilde{B})$. Therefore, $\widetilde{A}$ and $\widetilde{B}$ are isogenous over $L^{\prime}.$
 \end{proof}

Using Grothendieck-Messing theory, we further observe that the reduction map preserves the endomorphism orders. For the proof of the following lemma, we crucially need the abelian varieties to be defined over an unramified extension of $\Q_p$. 

\begin{lem}\label{lift}
Let $\widetilde{A}$ and $A$ be as above. We recall that $\widetilde{A}$ is a lift of $A$ defined over an unramified extension of $\Q_p$. Then $\End\widetilde{A} = \End A.$ 
\end{lem}

\begin{proof}
Let $\widetilde{\mathscr{A}}$ be the N\'eron model of $\widetilde{A}$. We denote by $\End(\widetilde{A}) = \End(\widetilde{\mathscr{A}}) = R$ and $\End(A) = S$. 
We know that taking the special fiber is a faithful functor from the category of abelian varieties over $W(k)$ to abelian varieties over $k.$ Therefore, we have a containment of orders $R\subset S$ inside $K$. We need to show that they are equal. Let $A[p^{\infty}]$ be the $p$-divisible group of $A$ and $\mathbb{D}(A[p^{\infty}])$ over $W(k)$ be its Dieudonn\'e module. By Grothendieck-Messing theory (see \cite{gm} Chapter 5, Theorem 1.6), the lifts of $A$ to $W(k)$ are in bijection with saturated submodules of $\mathbb{D}(A[p^{\infty}])$ over $W(k)$ that lift the filtration $\ker F\Mod p\subseteq\D(A[p^{\infty}]) \Mod p$ where $F$ is the Frobenius on $\mathbb{D}(A[p^{\infty}])$. Let $\mathcal{F}il$ be the  filtration corresponding to $\widetilde{A}$. We note that all the endomorphisms of $A$ that lift to $\widetilde{A}$ are the ones that leave $\mathcal{F}il$ invariant. Let $\widetilde\beta\in S.$ Since $R$ and $S$ are orders in $K$, their quotient is finite. Thus $n\widetilde\beta\in R$ for some positive integer $n.$  This implies that $\mathcal{F}il$ is invariant under $n\widetilde\beta$. As $\mathcal{F}il$ is saturated, we conclude that it is invariant under $\tilde{\beta}$ as well. Hence $\widetilde{\beta}\in R.$
\end{proof}     
The main result of this section is Proposition \ref{inj} which provides an injection between the isogeny class of $\widetilde{A}$ over $L^{\prime}$ to the isogeny class of its reduction mod $p$. The reader may wish to skip to the proposition directly. Its proof involves Fontaine's results on classification of finite flat group schemes, which we first review below. 
\subsection{Classification of finite flat group schemes of prime power order.} We follow the exposition and results in \cite{BC} Section 1, and state the results that we apply later. Let $\sigma:W(k)\rightarrow W(k)$ be the Frobenius morphism and denote by $D_k := W(k)[F,V]$, the Dieudonn\'e ring generated by two commuting elements, the Frobenius $F$ and the Verschiebung $V$ with the relations $FV=VF=p, Fa = \sigma(a)F,$ and $aV = \sigma(a)V$ for $a\in W(k).$ We recall that $p>2$.
\begin{defn}[Finite Honda systems]\label{honda} The category of \textit{finite Honda systems} over $W(k)$ consists of 
tuples $(\mathscr{L},\mathscr{M})$ where $\mathscr{L}$ and $\mathscr{M}$ are $D_k$-modules of finite $W(k)$-length such that the following are satisfied:

\begin{enumerate} 
\item $\mathscr{L}\subset \mathscr M$ is a $W(k)$-submodule of $\mathscr{M}$. 
  \item The $k$-linear map $\mathscr{L}/p\mathscr{L}\rightarrow \mathscr{M}/F\mathscr{M}$ is an isomorphism.
    \item The Verschiebung $V$ is injective on $\mathscr{L}$.
\end{enumerate}
\end{defn}
\begin{prop}(Fontaine, see \cite{BC} Lemma 1.3)\label{prop2fon} The category of finite Honda systems over $k$ is abelian. Let $\delta: (\mathscr{L}_1 ,\mathscr{M}_1)\rightarrow (\mathscr{L}_2,\mathscr{ M}_2)$ be a morphism in this category. Then the kernel is given by $\ker(\delta) := (\mathscr{L}^{\prime}, \mathscr{M}^{\prime})$ where $\mathscr{M}^{\prime} :=\ker(\mathscr{M}_1\rightarrow \mathscr{M}_2)$ and $\mathscr{L}^{\prime} :=\mathscr{ L}_1 \cap \mathscr{M}^{\prime}$. The cokernel is $\coker(\delta) := (\mathscr{L}^{\prime\prime}, \mathscr{M}^{\prime\prime})$ where $\mathscr{M}^{\prime\prime} :=\coker(\mathscr{M}_1\rightarrow \mathscr{M}_2)$ and $\mathscr{L}^{\prime\prime}$ is the image of the composition of maps of $W(k)$-modules
$\mathscr{L}_2\rightarrow \mathscr{M}_2 \rightarrow \mathscr{M} ^{\prime\prime}.$
\end{prop}    
\begin{prop}(Fontaine, see Theorem 1.4 \cite{BC})\label{prop1fon}
For $p>2$, there is an anti-equivalence between the category of finite flat commutative group schemes $G$ over $W(k)$ of $p$-power order and the category of finite Honda systems over $W(k)$.
\end{prop}
Under this equivalence, $G$ maps to $(\mathscr{L},\mathscr{M})$ where $\mathscr M$ is the  Dieudonn\'e module of the special fiber which we denote by $G_k$ and $\mathscr{L}$ is a $W(k)$ sub-module of $\mathscr M$ satisfying the three properties in Definition \ref{honda}. See, for example \cite{BC}, Section $1$ for a summary of Fontaine's results. We next record the following lemma.
\begin{lem}\label{sameker}
Let $\mathscr{A}$ be an abelian scheme over $W(k)$ with special fibre $A$.
Let $G_{1}$ and $G_{2}$ be two finite flat subgroup schemes of $\mathscr{A}$ over $W(k)$. We assume that the special fibre of both $G_1$ and $G_2$ is $H\subset A$ over $k$. Then $G_{1}$ equals $G_{2}$ as subgroup schemes of $\widetilde{\mathscr{A}}$.

\end{lem}
\begin{proof} 
  We can embed $G_{1}$, $G_{2}$ in  $\mathscr{A}[m]$, the kernel of multiplication by $m$ map, for some positive integer $m>1$. Thus, it is enough to consider the cases of group schemes with order prime to $p$ and $p$-power separately. For $(m,p)=1$, the reduction mod $p$ map induces an isomorphism on the $m$-torsion. Therefore, the prime-to-$p$ subgroups lift uniquely. When $m=p^r$, we associate the tuples $(\mathscr{L}, \mathscr{M}), (\mathscr{L}_1, \mathscr{M}_1)$ and $(\mathscr{L}_2, \mathscr{M}_2)$ to the group schemes $\widetilde{\mathscr{A}}[p^r], G_{1}$ and $G_{2}$ respectively, under the anti-equivalence in Proposition \ref{prop1fon}.  Since $G_1$ and $G_2$ reduce to the same group scheme over $k,$  we see that $\mathscr{M}_1 =\mathscr{M}_2$ as they correspond to the Dieudonn\'e module of the special fibre $H$. 
Furthermore, the tuples $(\mathscr{L}_1, \mathscr{M}_1)$ and $(\mathscr{L}_2, \mathscr{M}_2)$ are quotients of $(\mathscr{L}, \mathscr{M})$ in the category of finite Honda systems. By Proposition  \ref{prop2fon} we see that $\mathscr{L}_1 = \mathscr{L}_2$ as they are given by the image of the map $\mathscr{L}\rightarrow \mathscr{M}\rightarrow \mathscr{M}_1 = \mathscr{M}_2$ This shows that the group schemes $G_{1}=G_{2}$ as subgroup-schemes of $\widetilde{\mathscr{A}}[p^r]$. 
\end{proof}

We will now prove the following proposition.
\begin{prop}\label{inj}

Let $\widetilde{A}$ and $\widetilde{B}$ be abelian varieties over $L^{\prime}$ that induce the same CM type over $\overline\Q_p$. Moreover, assume that $\widetilde{A}$ and $\widetilde{B}$ both reduce mod $p$ to an abelian variety $A$ over $k$.
Then $\widetilde{A}$ and $\widetilde{B}$ are isomorphic over $L^{\prime}$.
\end{prop}
\begin{proof}
Since $\widetilde{A}$ and $\widetilde{B}$ have the same CM type, we get an isogeny $u: \widetilde{A}\rightarrow \widetilde{B}$ over $L^{\prime}$ by Lemma \ref{isogunrami}. As both $\widetilde{A}$ and $\widetilde{B}$ reduce to $A$ over $k$, the morphism $u$ reduces to an endomorphism $u_0:A\rightarrow A$ over $k$. By Lemma \ref{lift}, $u_0$ lifts to an endomorphism $v: \widetilde{A}\rightarrow \widetilde{A}$. Let $\overline{u}:\mathscr{A}\rightarrow\mathscr{B}$ and $\overline{v}:\mathscr{A}\rightarrow\mathscr{A}$ be the unique $W(k)$-morphisms extending $u$ and $v$ respectively to their N\'eron models over $W(k)$. By Lemma \ref{sameker}, we get $\ker(\overline{u})= \ker(\overline{v})$ as they both reduce to $\ker (u_0)$ mod $p$. Now, $\ker (u)\subseteq\ker (v)$ implies that there exists a map $ s:\widetilde{B}\rightarrow \widetilde{A}$ such that 
$s\circ u = v$. Similarly, the other way containment gives us a map $t:\widetilde{A}\rightarrow \widetilde{B}$ such that $t\circ v = u$. Finally, we check that $ s\circ t$ and   $t\circ s$ act as identity when we compose them with $v$ and $u$ respectively, thus completing the proof of the proposition.
\end{proof}

\section{Local conditions for the  endomorphism ring}\label{localcon}
In this section, we compute the local conditions for the endomorphism ring of an abelian variety $A$ over $\F_q$ with the conditions given in Section \ref{mainthm}. We recall that the Newton polygon of $A[p^{\infty}]$ has slopes $1/g$ and $(g-1)/g$. Moreover, we assume that the prime $p$ splits in $\o_K$ such that $K\otimes \Q_p = K^{\prime}\oplus K^{\prime}$ where $K^{\prime}$ is the unramified extension of $\Q_p$ of degree $g$.  We next find the local conditions of the orders that can occur as an endomorphism ring in the isogeny class of $A$ over $\F_{q}.$ 
\subsection{Local conditions at the prime $p$}
\begin{prop}\label{localcondn}
Let $A$ be an abelian variety over $\F_q$ with endomorphism algebra $K = \Q(\alpha)$ where $\alpha$ is the Weil-$q$ number of $A$. Let $\End(A) = R.$ Then the order $R$ is maximal at $p$. More precisely, $R\otimes \Z_p =\o_{K^{\prime}}\oplus \o_{K^{\prime}}$ where $\o_{K^{\prime}}$ is the maximal order in  $K^{\prime}$. 

\end{prop}

\begin{proof}
 Let $\D(A[p^{\infty}])$ denote the Dieudonn\'e module of $A[p^{\infty}]$. Note that $\D(A[p^{\infty}])$ admits a decomposition as a direct sum $\D(A[p^{\infty}])\otimes \Q_p = \D_{1}\times \D_2$ where $\D_1$ and $\D_2$ are modules over $K^{\prime}$. This further gives us a decomposition of the $p$-divisible group $A[p^{\infty}]\simeq \G_{1}\times \G_{2}$ over $\F_q$ where $\G_{1}$ and $\G_{2}$ have slopes $1/g$ and $(g-1)/g$ respectively (see Theorem 9.2 in \cite{abvaroort}), and $\G_{2}$ is the dual of $\G_1$.   Since $\G_{1}$ cannot be isogenous to $\G_{2}$, we see that $\End(\G_{1}\times \G_{2})\simeq \End(\G_{1})\times \End( \G_{2})$.  From Tate's theorem we get $R\otimes \Z_p \simeq \End A[p^{\infty}]$.  Therefore, it is sufficient to prove that $\End(\G_{1})$ and $\End(\G_{2})$ are maximal. We now follow the proof in \cite{AA}, Proposition $2.1$. Let $\G'$ be a $p$-divisible group isogenous to $\G_{1}$ such that $\End(\G')$ is maximal in its field of fractions. Let $j:\G_0\rightarrow \G_{1}$ be any arbitrary isogeny. Since $\G_{1}$ is isomorphic to a quotient of $\G_0$ by the kernel of  $j$, we will show that $\End(\G')$ preserves $\ker j,$ proving that $\End(\G')\subseteq \End(\G_{1})$. Since  $\G_{0}$ is a one dimensional $p$-divisible group, it has a unique subgroup of order $p^r$ for each $r\geq 1$.  Therefore, $\End(\G^{\prime})$ preserves the kernel of $j$. The same argument with the dual of $\G_{1}$ proves that $\End(\G_{2})$ is maximal.
 \end{proof}
\subsection{Local conditions away from $p$}

The above proposition proves that if $R$ is an order that occurs  as an endomorphism ring of an abelian variety isogenous to $A$, then $R$ must be maximal at $p.$ Therefore, it differs from any order that is maximal at $p$, at only finitely many primes away from $p.$ Let $l\neq p$ be any arbitrary prime. As $R$ contains the Frobenius endomorphism, the local order $\Z_l[\alpha]$ is contained in $R\otimes\Z_l$ where $\alpha$ is the Weil-$q$ number of $A.$ In fact, we have the following lemma that proves that we can always adjust the local conditions of an order at $l\neq p$ up to isogeny. Thus, it guarantees the existence of an abelian variety over $\F_{q}$ in the isogeny class with endomorphism ring $R$ with specified local conditions at primes other than $p$.  

\begin{lem}\label{adjustendoprimetop}
Let $A$ be an abelian variety over $\F_q$ with endomorphism algebra $K = \Q(\alpha)$ as above. Let $R$ be an order in $K$ containing $\Z[\alpha].$ Then we can find an abelian variety $A'$ isogenous to $A$ over $\F_q$ such that $\End(A')\otimes \Z_l =R\otimes \Z_l$ for all primes $l\neq p.$
\end{lem}

 The proof is already given in \cite{waterhouse}, Porism 4.3. For completeness, we present it here as well. We denote by $T_l(A)$ the Tate module of $A$ and let $V_l(A) = T_l(A)\otimes \Q_l.$

\begin{proof}
Let $\End(A) = S$. Without loss of generality, we assume $S\subseteq R$. By Proposition \ref{localcondn}, we know that $R\otimes \Z_p = S\otimes \Z_p = \o_{K^{\prime}}\oplus \o_{K^{\prime}}$. This shows that the index of $S$ in $R$ is not divisible by $p$.  Therefore, it is enough to consider primes $l\neq p$.  For primes larger than the index of $S$ in $R$, the two orders are locally the same. Let $l_1,\dots , l_m$ be finitely many primes where the two orders are not the same. According to Tate's theorem $S\otimes \Z_l$ preserves the lattice $T_l(A)$ in $V_l(A)$ for all primes $l.$ Since $K$ is commutative, $V_l(A)$ is a rank one module over $K\otimes \Q_l.$ Therefore, for any order in $K$, there exists a lattice with that prescribed order as its endomorphism ring. Let $T_{l_1}(A)$ be a sub-lattice of $T_l(A)$ that is preserved by the order $R\otimes \Z_{l_1}$. The quotient of the two lattices corresponds to a finite index subgroup of $A.$ We quotient $A$ by this subgroup to get an abelian variety $A'$ isogenous to $A$ such that $T_{l_1}(A') = T_{l_1}(A).$ Now repeat this process for all the other primes to finish the proof. 
\end{proof}

For each positive integer $n$, we count the isomorphism classes of abelian varieties in $I(A, \F_{q^n})$ with endomorphism ring $R_n$ characterized by the local conditions summarized below. This gives a lower bound on the size of $I(A, \F_{q^n})$. We note that the order $R_n$ is the ``smallest possible" at the primes $l\neq p$ and satisfies the necessary condition at the prime $p$ as proved in Proposition \ref{localcondn}.

\begin{defn}\label{smallestend}
Let $n$ be a positive integer. We define $R_n\subset K = \Q(\alpha)$ as the order satisfying the following conditions:
 \begin{enumerate}
     \item The order $R_n$ contains the order $\Z[\alpha^n, (q/\alpha)^n]$.
     \item It is maximal at $p,$ that is, $R_n\otimes \Z_p = \o_{K^{\prime}}\oplus \o_{K^{\prime}}$.
     \item At the prime $l,$ we have $R_n\otimes \Z_l = \Z_l[\alpha^n].$
 \end{enumerate}

\end{defn}

\section{Properties of abelian varieties of Lubin-Tate type.}\label{propring}
  In this section, we note certain properties of abelian varieties with endomorphism ring characterized in Definition \ref{smallestend} that will be useful to us in subsequent sections.
\subsection{Field of definition}

\begin{lem}(First property)\label{descend}
Let $n$ be a positive integer and $B$ be an abelian variety over $k$ with endomorphism ring $R_n$ as defined in Definition \ref{smallestend}. Then $B$ is defined over $\F_{q^n}$.
\end{lem}

\begin{proof}
Without loss of generality, we prove this result for $n=1.$ We write $R_1=R.$ As $\Z[\alpha,q/\alpha]\subset R\subset  \End^0(A)$, we can assume that $B$ is geometrically isogenous to an abelian variety $B^{\prime}$ defined over $\F_q$ with $q$-Frobenius $\alpha$. As the endomorphism of $B^{\prime}$ is already maximal at $p$, by Lemma \ref{adjustendoprimetop}, we may replace $B^{\prime}$ by an isogenous abelian variety over $\F_q$ that has endomorphism by the maximal order so that $R\subset \End B^{\prime}$. Let $\varphi_{B'}: B^{\prime}\rightarrow B$ be such an isogeny. Therefore, the kernel of $\varphi_{B'}$ is invariant under $R$. In particular, it is invariant under the action of $\alpha$, the $q$-Frobenius of $B^{\prime}$. This implies that the kernel is also Galois invariant under $\Gal(k/\F_q)$ because it acts by the Frobenius of $B^{\prime}$. This implies that $B$ is defined over $\F_q.$ 
 \end{proof}
 
\subsection{Central leaf}\label{centralleaf}
Let $D_k:=W(k)[F,V]$ denote the Dieudonn\'e ring as before and $\mathbb{D}(-)$ denote the contravariant Dieudonn\'e functor. For any positive integers $a$ and $b$, we write $\G_{a,b}$ for the $p$-divisible group over $k$ corresponding to the Dieudonn\'e module $\mathbb{D}(\G_{a,b}):= D_k/D_k(F^a - V^b)$ guaranteed by the Dieudonn\'e theory. We note that under this characterization, the $p$-divisible group $\G_{a,b}$ will have dimension $a$ and height $a+b.$ The Dieudonn\'e-Manin classification tells us that every $p$-divisible group over $k$ is isogenous to a product of $\G_{a,b}$ as $a$ and $b$ vary over positive integers.  In fact, in dimension one we have the following stronger uniqueness result.

\begin{lem}\label{uniquepdiv}
For any positive integer $h$, there exists a unique $p$-divisible group of dimension one and height $h$ over $k$. Consequently, considering its dual $p$-divisible group, the same is true for a $p$-divisible group of co-dimension one and height $h$.
\end{lem}

\begin{proof}
This result is well known. Let $\G$ be a $p$-divisible group of dimension one and height $h$ over $k$. Let $\D(\G)$ be its Dieudonn\'e module. Using the Dieudonn\'e-Manin classification, we know that $\G$ is isogenous to $\G_{1,h-1}$ over $k$. We need to show that they are isomorphic. We first note that $\D(\G)$ and $\D(\G_{1,h-1})$ have rank $h$ as $W(k)$-modules. We need to show that they are isomorphic as $D_k$-modules. We show this by constructing a basis for $\D(\G)$ over $W(k)$ that shows the uniqueness of the action of Frobenius and Verschiebung on $\mathbb{D}(\mathscr{G})$. By abuse of notation, we denote them by $F$ and $V$ as well. As $\mathscr{G}$ is isogenous to $\mathscr{G}_{1,h-1}$ over $k$, it is a simple Dieudonn\'e module. That is, it does not contain any $W(k)$-submodules stable under the action of $F$ and $V$. As $\D(\G)/F(\D(\G))$ is a one dimensional vector space over $k$ it is spanned by (the image of) an element, say $e_1$ in $\D(\G).$ Now, let $Fe_1 = e_2.$ We note that $Ve_2 = pe_1$ and $e_2$ is not in the $W(k)$-span of $e_1$. Otherwise, $\mathbb{D}(\mathscr{G})$ will have a one-dimensional $W(k)$-submodule spanned by $e_1$ which is stable under the action of $F$ and $V$. Next, we let $Fe_2=e_3$. Similarly, we see that $e_3$ cannot lie in the $W(k)$-span of $e_1$ and $e_2.$ Otherwise, we get a rank two sub-Dieudonn\'e module contained in $\mathbb{D}(\mathscr{G})$. We continue the same process until we get a basis of $\mathbb{D}(\mathscr{G})$ such that $F$ acts as $Fe_i = e_{i+1}$ for $1\leq i\leq h-1$. Finally, write $Fe_{h} = a_1e_1+a_2e_2+\dots +a_{h}e_{h}$. Since $e_1$ is not in the image of $F$ and $Fe_h=0$ in $\mathbb{D}(\mathscr{G})/F(\mathbb{D}(\mathscr{G}))$, we see that $a_1=p$. Using the fact that $VF=FV=p$, we see that $Fe_{h} =pe_1$. This gives us an isomorphism between $\G$ and $\G_{1,h-1}$. 
\end{proof}
 
\begin{defn}(Central leaves) Let $\G$ be a $p$-divisible group over $\F_q.$ Then the $\F_{q^{n}}$ points of the central leaf of $\G$ in $\mathcal{A}_{g,1}\otimes \F_q$ are defined as
$$c(\G)(\F_{q^n}) = \{B\in \mathcal{A}_{g,1}(\F_{q^n})\mid \G\otimes k \simeq B[p^{\infty}]\otimes k\}.$$
Let $A\in \mathcal{A}_{g,1}(\F_q)$ be a principally polarized abelian variety. The central leaf of $A$ denoted as $c(A)$ is defined as the central leaf associated to its $p$-divisible group $A[p^{\infty}]$. We note that $c(A)$ cuts out a closed sub-variety in the Newton stratum of $A$.
\end{defn}

\begin{lem}\label{uniquecf}(Second property)
Let $n$ be a positive integer. Let $B$ be an abelian variety over $\F_{q^n}$ satisfying the conditions mentioned at the beginning of Section \ref{localcon}. Let $\End(B) = R_n\subset K$. Then $B$ lies in the central leaf of the $p$-divisible group $\G_{1,g-1}\times \G_{g-1,1}$.
\end{lem}
\begin{proof}
Applying Proposition \ref{localcondn} to the abelian variety $B$, we see that $\End(B[p^{\infty}])\otimes k  = \o_{K^{\prime}}\oplus \o_{K^{\prime}}$. This implies that its Dieudonn\'e module $\D(B[p^{\infty}])\otimes k$ splits into a direct sum $\D_1\oplus \D_2$ such that $\End \D_1$ and $\End \D_2$ are maximal at $p.$ This gives the desired decomposition of $B[p^{\infty}]\otimes k$ into a product $\G_{1}\times \G_{2}$.  By the Dieudonn\'e-Manin classification $\G_{1}\times \G_{2}$ is isogenous to $\G_{1,g-1}\times \G_{g-1,1}$. Let $\psi$ be such an isogeny. Since the endomorphism ring preserves the kernel of $\psi$, it must also factor into a direct summand. This shows that $\G_{1}$ (respectively, $\G_{2}$) is isogenous to $\G_{1,g-1}$ (respectively, $\G_{g-1,1}$). By Lemma \ref{uniquepdiv}, we see that they are isomorphic. This shows that the isogeny $\psi$ is an isomorphism and $B$ lies in the central leaf of $\G_{1,g-1}\times \G_{g-1,1}$.
\end{proof}

\section{Isogeny classes and the class group}\label{count}
 Let $R\subset K$ be any order. Let $I\subset K$ be an invertible $R$ module. That is, $I$ is a fractional ideal of $R$ such that there exists some fractional ideal  $J$ with $IJ = R$. We say that two fractional ideals, $I_1$ and $I_2$ are equivalent if and only if there exists some $a\in K^{\times}$ such that $I_1 = aI_2.$ 
 
 \begin{defn}(Class group of an order)
Under the equivalence mentioned above,  we define the equivalence classes of  fractional ideals of $R$ to be the \textit{class group of $R$}, denoted as $Cl(R).$
 \end{defn}
 
 The heart of the rest of the paper lies in the following proposition that gives us a convenient characterization of (isogenous) abelian varieties over $L^{\prime}$ with endomorphism ring $R_n$ in Definition \ref{smallestend}. Here we recall the setup and notation at the beginning of Section \ref{sec1}. In particular, we return to the set up in characteristic zero with $\widetilde{A}$ an abelian variety over $L= W(\F_q)(1/p)$. We denote by $A$ over $\F_q$ its reduction mod $p$. We recall from Section \ref{sec1} that we fixed a primitive CM type $\Phi$ on $K$ which is determined by $\widetilde{A}$ after choosing an embedding $L'\hookrightarrow \C$.
\begin{prop}\label{clinj}
Let $\Phi$ be the above fixed primitive CM type on $K$. Then for any positive integer $n$, there exists an injective map from the class group of $R_n$ to the set of isomorphism classes of abelian varieties over $L^{\prime}$ isogenous to $\widetilde{A}$ with endomorphism ring $R_n \subset K.$

\end{prop}

\begin{proof}

Let $I$ be a representative of an equivalence class in $Cl(R_n).$ We recall that $\Phi$ gives an isomorphism $K\otimes \C = \C^{\Phi}$ where  $\C^{\Phi}$ denotes $g$ copies of $\C$ indexed by the embeddings in $\Phi.$ Since $I$ is a lattice in $K,$ we can embed it inside $\C^{\Phi}$. We refer to the image of $I$ under this embedding as $\Phi(I).$ The complex torus  $\C^{\Phi}/\Phi(I)$ admits a polarization (see Section \ref{polarization}) and therefore we get an abelian variety over $\C$ which we denote by $A_{\Phi(I)}$. Since $\Phi$ is primitive, $A_{\Phi(I)}$ is simple and by construction $\End(A_{\Phi(I)}) = R_n.$  Furthermore, since $A_{\Phi(I)}$ is a CM abelian variety, it descends, with its entire endomorphism algebra, to a number field. Now, if the abelian variety had CM by the maximal order, then the abelian variety is defined over the maximal unramified extension of the reflex field $K^{*}$ and it would be defined over an unramified extension of $\Q_p.$ In our case, however, the field of definition might be a larger field extension that is a ramified extension of $\Q_p$. We first rule out this possibility for the abelian variety $A_{\Phi(R_n)}.$  We note that $A_{\Phi(R_n)}$ is isogenous to the abelian variety $A^{\prime}$ corresponding to the lattice $\o_K$ in $K$. Now, $A^{\prime}$ has CM by a maximal order over $\overline{\Q}_p.$ Let $F$ be the unramified extension of $\Q_p$ where  $A^{\prime}$ is defined.  The kernel of the isogeny has order equal to the index of $R_n$ in $\o_K$. This, along with the fact that $R_n$ is maximal at $p$ implies that we can embed the kernel inside $A^{\prime}[m]$ for some $(m,p)=1$. The kernel is defined over some extension $F^{\prime}$ of $F$. This must be unramified because  the inertia group at the place $\bar v$ acts trivially on the $m$-torsion and therefore on the subgroup as well. Here, $\bar{v}$ refers to the extension of a place $v$ of $F^{\prime}$ to its algebraic closure. This descends $A_{\Phi(R_n)}$ to an unramified extension of $\Q_p.$ Now let $I$ be a fractional ideal of $R_n$ and let $J:=I\o_K$ be its extension to a fractional ideal of $\o_K$. As $I$ is invertible, $[I:J]=[R_n:\o_K]$. As the abelian variety corresponding to $J$ has endomorphism by the maximal order, the same argument works and descends $A_{\Phi(I)}$ to an unramified extension of $\Q_p.$  All the abelian varieties constructed above are isogenous to $\widetilde{A}_{L^{\prime}}$ over $L^{\prime}$ by Lemma \ref{isogunrami}.  Finally, as the index $[I:J]=[aI:aJ]$ for any $a\in K^{\times}$, the above argument is independent of the choice of the representative. Under the equivalence of  ideals in $Cl(R_n)$, two equivalent ideals would correspond to isomorphic lattices and give us isomorphic abelian varieties over $\overline\Q_p$, and by the argument in Lemma \ref{isogunrami}, over $L^{\prime}$ as well.
\end{proof}

\begin{cor}\label{injclgrpnopol}
Let $n$ be any positive integer. Then there exists an injective map from the class group $Cl(R_n)$ to the set of isomorphism classes of abelian varieties over $\F_{q^n}$ isogenous to $A$ (without any condition on polarization).
\end{cor}
\begin{proof}  By Proposition \ref{inj}, the reduction mod $p$ map from the set of isomorphism classes of abelian varieties isogenous to $\widetilde{A}$ over $L^{\prime}$ to that of $A$ over $k$ is injective. In the proof of Proposition \ref{clinj}, we can choose the place $v$ such that $\End^0(\widetilde{A})=K=\End^0(A)$.  Moreover, for any $n$, by Lemma \ref{descend} the abelian variety constructed using an ideal $I\subset R_n$ is defined over $W(\F_{q^n})(1/p)$ and therefore the reduction mod $p$ of its N\'eron model descends to $\F_{q^n}$. By Lemma \ref{lift}, its endomorphism ring is $R_n$ and therefore has $\alpha^n\in K$ as its Frobenius morphism. 
 \end{proof}
\section{Estimates}\label{estimates}
The goal of this section is to find the asymptotics for the class number of $Cl(R_n)$. We prove the following estimate.
\begin{prop}\label{estimate}
We assume that $1,\alpha =\alpha_1,$ and its conjugates $\alpha_2, \dots \alpha_g$, are multiplicatively independent. Then for a positive density set of positive integers $n$, the class group $|Cl(R_n)| = q^{n{((g^2-2g)/2+o(1))}}.$

\end{prop}
\begin{proof}
Using the analytic class number formula for orders in $\o_K$ (see \cite{classnumberformula}, Theorem 1.1), we can approximate the size of the class group of an order by the square root of its discriminant. Let $\alpha_i$, where $1\leq i\leq 2g$ and $\alpha_1 =\alpha$ be the conjugates of $\alpha.$ For each $i$, we write $\alpha_i =q^{1/2}(\cos{\theta_i}+\iota\sin{\theta_i})$. Then the discriminant of the monogenic order $\Z[\alpha^n]$ is given by $$\disc(\Z[\alpha^n])=\prod_{1\leq i<j\leq 2g}({\alpha^n_i} - {\alpha^n_j})^2 =q^{\frac{2ng(2g-1)}{2}}\prod_{1\leq i<j\leq 2g}(\cos(n\theta_i) + \iota\sin(n\theta_i) - \cos(n\theta_j) - \iota\sin(n\theta_j))^2 $$ Since $\disc(\Z[\alpha^n])= [\Z[\alpha^n]: R_n]^2\disc(R_n)$, it is enough to estimate the index of $\Z[\alpha^n]$ in $R_n.$ We note that the index is $1$ at all primes $l\neq p$ (see Definition \ref{smallestend}, point (3)). Therefore, it is sufficient to estimate the index after tensoring the orders with $\Z_p$.
We do this below.
\subsection{Estimating the index at the prime $p$} We write  $ \Z[\alpha]\otimes \Z_p\subset \o_K\otimes\Z_p = \o_{K^{\prime}}\oplus \o_{K^{\prime}}$. Under this embedding, the image of $\alpha$ is $(\beta, \gamma)$ where $v_{q}(\beta) = 1/g$ and $v_{q}(\gamma)= (g-1)/g$.  The ring $\Z_p[(\beta, \gamma)]$ is monogenic and is isomorphic to $\Z_p[x]/(f_1f_2)$ where $f = f_1f_2$ is the factorization of the minimal polynomial $f$ of $\alpha$ in $\Z_p.$ Let $\beta_i$ and $\gamma_i$ be the roots of $f_1$ and $f_2$ with $\beta_1 = \beta$ and $\gamma_1=\gamma$ and $1\leq i\leq g$. 

Now, for any positive integer $n$, we are reduced to finding the index of $\Z_p[\alpha^n] = \Z_p[(\beta^n, \gamma^n)]$ in $R_n\otimes\Z_p=\o_{K^{\prime}}\oplus \o_{K^{\prime}}$. We compute this as follows: consider the monogenic ring $\Z_p[x]/g_1g_2 := \Z_p[\delta]$ where $\delta = (\beta/q^{1/g},\gamma/q^{(g-1)/g})$ and the roots of $g_1$ and $g_2$ are given by $\beta^{\prime}_i:=\beta_i/q^{1/g}$ and $\gamma^{\prime}_i:=\gamma_i/q^{(g-1)/g}$ respectively where $1\leq i\leq g$. 
Thus, we have the following containment of orders: $$\Z_p[\alpha^n]\subseteq \Z_p[\delta^n]\subseteq \o_{K^{\prime}}\oplus \o_{K^{\prime}}.$$

We first show that for a positive proportion of positive integers $n$, it is enough to compute the index of $\Z_p[\alpha^n]$ in $ \Z_p[\delta^n]$. That is, for most $n$, we do not get any contribution to the main term of the count from the index of $\Z_p[\delta^n]$ in $\o_{K^{\prime}}\oplus \o_{K^{\prime}}$. We justify this by showing that for most positive integers this index is constant up to some unit in $\o_{K'}$. Indeed, the discriminant of $\o_{K^{\prime}}\oplus \o_{K^{\prime}}$ is constant and does not depend on $n$, while the discriminant of $\Z_p[\delta^n]$ can be written as:
$$
\disc (\Z_p[\delta^n]) 
=\disc (\Z_p[\delta]) \cdot c(n)^2
$$
where $$c(n)=\prod_{1\leq i< j\leq g} \frac{({\beta^{\prime}_i}^n- {\beta^{\prime}_j}^n)}{(\beta^{\prime}_i- \beta^{\prime}_j)}\frac{({\gamma^{\prime}_i}^n- {\gamma^{\prime}_j}^n)}{(\gamma^{\prime}_i- \gamma^{\prime}_j)}\prod_{1\leq i\leq j\leq g}\frac{({\beta^{\prime}_i}^n- {\gamma^{\prime}_j}^n)}{(\beta^{\prime}_i- \gamma^{\prime}_j)}\\
$$ For a positive density set of integers $n$, we note that $v_q(c(n))=0$, that is, the discriminant of $\Z_p[\delta^n]$ and $\Z_p[\delta]$ differ only by a unit in $\o_{K'}$. We justify this for one of the terms in the expression of the product, and the argument is similar for the other terms. Consider the following term: 
$$c_{\beta,i,j}(n):=\frac{{(\beta^{\prime}_i}^n- {\beta^{\prime}_j}^n)}{(\beta^{\prime}_i- \beta^{\prime}_j)}$$
We know that $\beta_i^{\prime}$ and $\beta_j^{\prime}$ are units in $\o_{K'}$. Write $a\in \F_q^{\times}$ for the residue class of $\beta_i^{\prime}/\beta_j^{\prime}$ mod $p$. Then the residue class of $\frac{c_{\beta,i,j}(n)}{{\beta_j^\prime}^{n-1}}$ equals $1+a+\dots +a^{n-1}$. Thus, if $a=1$, then for all $n$ co-prime to $p$, this sum is not zero and $c_{\beta,i,j}(n)$ is a unit in $\o_{K'}$. Similarly, if $a\neq 1$, then this sum is not zero when $n\mid \ord(a)$. As long as $n$ is not a multiple of any of the divisors of $q-1$, we see that $c_{\beta,i,j}(n)$ is a unit in this case as well. From this discussion, we conclude that for positive integers $n$ such that $\gcd(n,q(q-1))=1$, the term $c(n)$ is a unit. The density of such integers is $\phi(q(q-1))/q(q-1)$ where $\phi$ denotes the Euler totient function.

For $n$ such that $c(n)$ is a unit, it is enough to compute the index of $\Z[\alpha^n]$ in $\Z[\delta^n]$. This
 is straightforward using the following relation:
$$\disc \Z_p[\delta^n]\cdot [\Z_p[\alpha^n]:\Z_p[\delta^n]]^2 = \disc \Z_p[\alpha^n] $$

The index is given by the expression for $n=1$ and can be raised to the power $n$ for any positive integer $n$.

  $$[\Z_p[\alpha]:\Z_p[\delta]]= \frac{\prod_{1\leq i< j\leq g}(\beta_i - \beta_j)(\gamma_i - \gamma_j)\prod_{1\leq i,j\leq g}(\beta_i - \gamma_j)}{\prod_{1\leq i< j\leq g} (\beta_i /q^{1/g}- \beta_j/q^{1/g})(\gamma_i/q^{(g-1)/g} - \gamma_j/q^{(g-1)/g)})\prod_{1\leq i,j\leq g}(\beta_i /q^{1/g}- \gamma_j/q^{(g-1)/g})}$$

which equals $(q^{1/g})^{g(g-1)/2}\cdot (q^{(g-1)/g})^{g(g-1)/2}\cdot(q^{1/g})^{g^2}  = q^{(g^2+g)/2}$ up to some constant. Hence, for a positive density set of positive integers, we get $[\Z[\alpha^n]: R_n] = [\Z_p[\alpha^n]:\Z_p[\delta^n]] = q^{n(g^2+g)/2}$ up to some constant. 
\subsection{Final estimate}\label{unpolri} 
As in the previous section, we take $n$ such that $\gcd(n, q(q-1))=1$. The expression for the discriminant of $R_n$ (up to some constant) is then given by:

$$(\disc(R_n))^{1/2}=\frac{(\disc(\Z[\alpha^n]))^{1/2}}{[\Z[\alpha^n]:R_n]}=\frac{q^\frac{ng(2g-1)} {2}\prod_{i<j\leq 2g}(\cos(n\theta_i) + \iota\sin(n\theta_i) - \cos(n\theta_j) - \iota\sin(n\theta_j))^2}{q^{n(g^2+g)/2}}$$
 
Let $T(n):=\prod_{i<j\leq 2g}(\cos(n\theta_i) + \iota\sin(n\theta_i) - \cos(n\theta_j) - \iota\sin(n\theta_j))^2$ be the term with sines and cosines in the above expression. We note that $T(n)$ is absolutely bounded above. We bound $|T(n)|$ below as in the proof of Proposition 3.6 in \cite{tsimshank}. We need a Lemma.

\begin{lem}\label{equidistribute}
For $(n, q(q-1))=1$, the sequence $(n\theta_1,\dots ,n\theta_g)$ is equidistributed in $[0,2\pi]^g$.
\end{lem}
\begin{proof}
  We now check Weyl's equidistribution criterion  for the sequence $(n\theta_1,\dots ,n\theta_{g})$ with $n$ such that $\gcd(n,q(q-1))=1$. Let $S_N =\{1\leq n\leq N\mid \gcd(n,q(q-1))=1\}$. Let $m=(m_1,\dots,m_g)\in \Z^r$ and $x =(x_1,\dots, x_g)$ where $x_i= \theta_i/2\pi$. When $q,\alpha,\alpha_2, \dots ,\alpha_g$, are multiplicatively independent we see that $(2\pi,\theta_1,\dots\theta_g)$ are $\Q$-linearly independent. Therefore, $\langle m, x\rangle :=m_1x_1+\cdots +m_gx_g\notin \Q$. We compute:
$$\lim_{N\to \infty}\frac{1}{|S_N|}\sum_{n\in S_N} e^{2\pi i n \langle{m,x}\rangle} = \lim_{N\to \infty}\frac{1}{|S_N|}\sum_{a\in (\Z/q(q-1)\Z)^{\times}}\sum_{\substack{n\in S_N\\ n\equiv a \Mod N}} e^{2\pi ni \langle{m,x}\rangle } \to 0$$

Indeed, for a fixed $a\in (\Z/q(q-1)\Z)^{\times}$, writing $n=a+kl$ and varying $0\leq k\leq \lfloor{(N-a)/l}\rfloor:=M$  we compute the inner sum as a geometric sum: 
$$\sum_{\substack{n\in S_N\\ n\equiv a \Mod N}} e^{2\pi ni \langle{m,x}\rangle }  = \sum_{k=0}^{M} e^{2\pi i (a+kl)\langle{m,x\rangle}} = e^{2\pi i a\langle {m,x}\rangle} \frac{1-e^{2\pi i l(M+1)}}{1-e^{2\pi i l}}$$
Taking absolute values, we see that the geometric sum is bounded above independent of $M$ (and hence, of $N$) for each residue class $a\in (\Z/q(q-1)\Z)^{\times}$. As $|S_N|=\phi(q(q-q))N/ q(q-1)$, we see that the above limit goes to $0$. By \cite{weylequidistribution} Theorem 6.2 in Chapter 1, the lemma follows.
\end{proof}

Write $|T(n)| = \prod_{i<j\leq 2g}|2\sin(n(\theta_i-\theta_j)/2)|$. By Lemma \ref{equidistribute}, for any fixed $\epsilon>0$, the argument $n(\theta_i-\theta_j)/2$ lies in $(\epsilon, \pi-\epsilon)$ for a positive density set of positive integers $n$. Therefore, each term $|\sin(n(\theta_i-\theta_j)/2)|\geq \sin(\epsilon)>1/N_{\epsilon}$. We get that for all $n>N_{\epsilon}$, this term in the product is greater than $1/n$ for a positive density set of integers $n$. Thus, we see that the size of the class group $Cl(R_n)$ is well approximated by the leading term
$q^{\frac{ng(2g -1)}{2}- \frac{n(g^2+g)}{2}} = q^{n(g^2-2g)/2}$.
\end{proof}

\begin{cor}\label{unpol}
For a positive density set of positive integers $n$, the size of the isogeny class of $A$ (without any conditions on polarization) in characteristic zero is greater than or equal to $q^{n(\frac{(g^2-2g)}{2}+o(1))}$.
\end{cor}
\begin{proof}
This follows from Corollary \ref{injclgrpnopol} and the calculation in Proposition \ref{estimate}.
\end{proof}

\section{Polarization}\label{polarization}

 So far, we have an estimate for the size of the isogeny class without specifying any conditions on the polarization.  To restrict this count to the isogeny class in $\mathcal{A}_{g,1},$ we need to impose further conditions for polarization. In \cite{polrihowe}, Howe describes the polarization data for ordinary abelian varieties that is nicely summarized in \cite{tsimshank}, Section 3.1. However, since the data is described purely in characteristic zero and are independent of the type of abelian variety, we use the same description in our case that we describe below.  

\subsection{Polarization data.}
 We work with the CM type $\Phi$ on $K$ as before. For any element $x\in K$, we denote by $\overline{x}$ the complex conjugation (induced by $K$) on $x$. The polarization on an abelian variety corresponding to the equivalence class of an  ideal $I$ in $Cl(R_n)$ is given by a $\Q$-valued bilinear form $\Psi$ on $K$:

$$\Psi: (x,y)\mapsto \tr_{K/\Q}(u x\overline y)$$ satisfying the following properties.
\begin{enumerate}
\item The form $\Psi$ restricted to $I$ is valued in $\Z$.
    \item  The element $u\in K$ is purely imaginary. 
 \item For all $\varphi\in \Phi$, we have $\varphi(u)/i$ is positive. 
 \end{enumerate}
We note that if an ideal $I$ is dual to itself with respect to the above form, then its corresponding abelian variety will be principally polarized. We denote by $I^{\vee}$ the dual of the fractional ideal $I$ with respect to $\Psi$. We have the following proposition. 
\begin{prop}\label{kernorm}
Let $R_n^{+}$ denote the intersection of $R_n$ with the maximal totally real field $K^+\subset K.$ Then the subset of $I(\widetilde{A}, L^{\prime})$ with endomorphism ring $R_n$ is empty or in one-to-one correspondence with the kernel of the norm map:
$$N:Cl(R_n)\rightarrow Cl^{+}(R_n^{+})$$ Here $Cl^{+}(R_n^{+})$ denotes the narrow class group of the order $R_n^{+}$ and the norm map refers to the map $I\mapsto I\overline{I}$.
\end{prop}

\begin{proof}
The proof of this proposition goes through verbatim as in Lemma 3.8 in \cite{tsimshank}.
 \end{proof}

\begin{lem}\label{finalestimate}
For a positive density set of integers $n,$ we have 
$\frac{|Cl(R_n)|}{|Cl^{+}(R_n^{+})|} =q^{(n(g+1)(g-2)/4+o(1))}$

\end{lem}

\begin{proof}
We first compute the leading term of $|Cl^+(R_n^+)|$. Let $\lambda^{(n)}$ be the element  $\alpha^n+q^n/\alpha^n$. We denote the conjugates of $\lambda^{(n)}$ by $\lambda^{(n)}_i$ for $1\leq i\leq g$ and $\lambda^{(n)}_1:=\lambda^{(n)}$.  As before, we write $\disc(\lambda^{(n)}) = [\Z[\lambda^{(n)}]:R_n^+]^2\disc(R_n^+)$ and  estimate the index $[\Z[\lambda^{(n)}]: R_n^+].$ For simplicity, we do the computation with $n=1$ as for a general $n$, the estimate is raised to the power of $n$. We write $\lambda^{(1)}_i = \lambda_i$ for $1\leq i\leq g$, $\lambda_1:=\lambda$ and $R^+_1 := R^+$. It is enough to compute the index of $\Z[\lambda]$ in $R^+$ at the prime $p$. We recall $K^{\prime}$ is the unramfied extension of $\Q_p$ of degree $g$. As $R^+\otimes\Z_p = \o_{K^{\prime}}$, we follow the same argument as in the proof of Proposition \ref{estimate} to compute this. Therefore, for a positive proportion of positive integers $n$ satisfying certain co-primality conditions, the index is given by the following expression: $$[\Z[\lambda]: R^+]=\frac{\prod_{1\leq i<j\leq g}(\lambda_i - \lambda_j)}{\prod_{1\leq i<j\leq g}(\lambda_i/q^{1/g} - \lambda_j/q^{1/g})} = q^\frac{{(g-1)}}{2}$$
The discriminant of the monogenic ring $\Z[\lambda]$ is a straightforward calculation given in Lemma 3.8 in \cite{tsimshank} where the leading term is given by $q^{g(g-1)/2}.$ Now, the leading term of $\disc(R^+)$ is approximated by dividing this discriminant by the square of the index computed above. The leading term of $|Cl^+(R^+)| = (\disc(R^+))^{1/2}$ equals $q^{\frac{g(g-1)}{4} - \frac{(g-1)}{2}}=q^{(g^2-3g+2)/4}$. Combining this (raised to the power of $n$) with the estimate of $|Cl(R_n)|$ in Proposition \ref{estimate}, we get the desired leading term for $|Cl(R_n)/Cl(R_n^+)|$. The terms involving sines and cosines in the expression of $|Cl^+(R_n)/Cl(R_n^+)|$ are absolutely bounded above. They are bounded below by $1/n$ for a positive density set of integers $n$ by the same argument done in Section \ref{unpolri}.  This completes the proof of Lemma \ref{finalestimate}.
\end{proof}
We note that a polarization on $\widetilde{A}$ extends to a polarization on its N\'eron model and reduces to a polarization on its reudction mod $p$ as well. Therefore, the count in characteristic zero restricted to principally polarized varieties gives us the right estimate in characteristic $p$.
It remains to prove that the kernel of the Norm map in Proposition $\ref{kernorm}$ is non-empty. That is, the subset of $I(\widetilde{A}, L^{\prime})$ with endomorphism ring $R_n$ is not empty. The following proposition, along with the estimate in Lemma \ref{finalestimate}, completes the proof of Theorem \ref{mainthm}.

\begin{prop}
 We assume that $1,\alpha =\alpha_1,$ and its conjugates $\alpha_2, \dots \alpha_g$, are multiplicatively independent. Then for a positive density set of integers $n$, there exists a  principally polarized abelian variety over $L^{\prime}$ with endomorphism ring $R_n$.

\end{prop}

\begin{proof}
 Let $h_n$ denote the minimal polynomial of $\alpha^n+(q/\alpha)^n$ and $h_n^{\prime}$ be its derivative. Let $u_n =(\alpha^n-(q/\alpha)^n h^{\prime}_n(\alpha^n+(q/\alpha)^n))^{-1}$. We note that $\alpha^n-(q/\alpha)^n$ is purely imaginary while $h^{\prime}_n(\alpha^n+(q/\alpha)^n)$ is totally real. Hence $u_n$ is purely imaginary for all $n$. By (the proof of) Proposition 9.4 in \cite{polrihowe}, the trace dual of $\Z[\alpha^n , (q/\alpha)^n]$ is $u_{n}^{-1}\Z[\alpha^n , (q/\alpha)^n]$. Therefore, if we consider the bilinear form $$(x,y)\mapsto \tr_{K/\Q}(u_nx\overline{y})$$
 then the dual of $\Z[\alpha^n , (q/\alpha)^n]$ with respect to this form is itself. Furthermore, the above form induces a polarization on the elements of $Cl(R_n)$ when $\varphi(u_n)/i$ is a positive real number for every $\varphi\in \Phi.$ The argument given in Proposition $3.6$ of \cite{tsimshank} shows that for a positive proportion of $n$, $\varphi(u_n)/i$ is positive under the assumption that the set $\{q^{1/2}, \alpha_1, \dots ,\alpha_g\}$ is multiplicatively independent.  We assume that $n$ is an integer such that $u_n$ satisfies appropriate conditions so that we get a polarization on the elements of $Cl(R_n).$  Without loss of generality, we assume $n=1$ as the argument remains the same for any $n\geq 1$. Let $\widetilde{A}/L^{\prime}$ be the abelian variety corresponding to the order $R_1:=R$ under the map in Proposition \ref{clinj}. We note that $\widetilde{A}$ has endomorphism ring equal to $R$. Further, recall that $p$ splits  as a product to two primes in $\o_K$. We let $\p_1$ and $\p_2$ be the prime ideals that we get by intersecting them with $R$. Since $R$ is locally  maximal at the primes lying over $p$, the prime ideals $\p_1$ and $\mathfrak{p}_2$ are invertible and give us representatives of elements in $Cl(R)$. Note that $\Z[\alpha , (q/\alpha)]$ and $R$ locally agree everywhere  except at $\p_1$ and $\p_2,$ and the same is true for $R^{\vee}$ and $\Z[\alpha , (q/\alpha)]^\vee =\Z[\alpha, (q/\alpha)]$. Thus, we can write $R^\vee =\p_1^k\p_2^m$ for some positive integers $m$ and $k.$ The dual $R^\vee$ corresponds to the dual abelian variety $\widetilde{A}^{\vee}$. By construction, there exists an isogeny $v:\widetilde{A}\rightarrow \widetilde{A}^{\vee}$. We know that the $p$-divisible group of $\widetilde{A}$ admits a decomposition $\mathscr{G}_{1/g}\oplus\mathscr{G}_{(g-1)/g}$ according to the splitting of the prime $p$ in $\o_K$. As $v$ also induces an isogeny of the respective $p$-divisible groups, we see that the kernel of $v$ is a finite flat subgroup of $\widetilde{A}$ corresponding to the direct sum $\o_{K'}/\mathfrak{p}_1^k\oplus \o_{K'}/\mathfrak{p}_2^m$. The finite flat subgroup to which this quotient corresponds is the direct sum of $\overline{\omega}^k$-torsion of $\mathscr{G}_{1/g}$ and $\overline{\omega}^m$-torsion of $\mathscr{G}_{(g-1)/g}$  where $\overline{\omega}$ is the uniformizer of $\o_{K^{\prime}}$. Since the polarization map has a kernel that is self dual, we see that $m=k.$  Therefore, the kernel equals the $p^{m}$-torsion of $\widetilde{A}$ which shows that $\widetilde{A}$ is principally polarized as $\widetilde{A}^{\vee}\simeq \widetilde{A}/\widetilde{A}[p^{m}] \simeq \widetilde{A}$.
 \end{proof}

 \bibliography{AV}
\bibliographystyle{amsalpha}


\end{document}